\documentclass{amsart}
\usepackage{amsmath}
\usepackage{amssymb}
\usepackage{graphicx}
\input xy
\xyoption{all}
\newtheorem{theorem}{Theorem}

\newtheorem{lemma}[theorem]{Lemma}

\theoremstyle{definition}

\newcommand{\Char}{\mbox{\rm char}}
\newcommand{\GL}{\mbox{\rm GL}}
\newcommand{\SL}{\mbox{\rm SL}}

\newcommand{\ts}[1]{\langle #1\rangle}
  
\begin{document}
    \title[Corrigendum to the proof of Theorem 3.4]{Corrigendum to the proof of one theorem in the article ``On multiplicative subgroups in division rings" [J. Algebra and Its Applications Vol. 15,\\ No. 3 (2016) 1550050 (16 pages)]}
    \author[Bui Xuan Hai]{Bui Xuan Hai$^{1,2}$}
    \address{[1] University of Science, Ho Chi Minh City, Vietnam; [2] Vietnam National University, Ho Chi Minh City, Vietnam.}
   \email{bxhai@hcmus.edu.vn}

\keywords{ maximal subgroups, subnormal subgroups,  $FC$-elements. \\
\protect \indent 2010 {\it Mathematics Subject Classification.} 16K20.}

 \maketitle

 \begin{abstract} In this note, the correction to the proof of one theorem in some our previous paper will be given. 
 \end{abstract}

\bigskip

In the proof of \cite[Theorem 3.4]{hai-tu} there is the following mistake: in row $-12$, page $1650050-11$, we affirmed that by applying \cite[Theorem 5]{hai-ngoc}, it follows that  $M$ is a (locally abelian)-by-finite group. But, in the reality,  this argument is not valid as it was well remarked in recently accepted paper \cite{fallah}. Then, in the main theorem of \cite{fallah}, the author has proved  the similar result as our in \cite[Theorem~ 3.4]{hai-tu}, but for more narrow class of division rings. In fact, he proved the result for centrally finite division rings, while our theorem holds for locally finite division rings.  Firstly, we must confirm that \cite[Theorem~ 3.4]{hai-tu} is always true. Moreover, this result was generalized in \cite[Theorem~ 3.1]{hai-khanh} for more general circumstance. Hence,  \cite[Theorem~ 4.1]{fallah}, and \cite[Theorem 3.4]{hai-tu} both are particular cases of \cite[Theorem~ 3.1]{hai-khanh}. Our purpose in this note is to give the another proof of  \cite[Theorem 3.4]{hai-tu}. The idea of this proof is very different from our in the proof of \cite[Theorem~ 3.1]{hai-khanh}, and has no any connection with the proof of \cite[Theorem 4.1]{fallah}.

For the next use, we need some lemmas.

\begin{lemma}\label{lem:1} Let $D$ be a locally finite division ring with the center $F$, and $G$ a subgroup of $D^*$. If $G'$ is radical over $F$, then $G'$ is locally finite and either  $G'\cong\SL(2,5)$ or $G$ is solvable.
\end{lemma} 
\begin{proof}  In the first, we show that if an element $x\in D'$ is radical over $F$ then $x$ is periodic. Indeed, assume that $x=[x_1,y_1]\ldots [x_k,y_k]$ and $x^n=a\in F$ for some positive integer $n$. The division subring $K=F(x_1,y_1,\ldots,x_k,y_k)$ generated by $x_1,y_1,\ldots,x_k,y_k$ over $F$ in $D$ is a finite dimensional vector space over $F$, say, of degree $m=[K:F]$. Hence,
	$$1=N_{K/F}(x)^n=N_{K/F}(x^n)=N_{K/F}(a)=a^m,$$
where $N_{K/F}$ denotes the norm of $K$ to $F$. So $x^{nm}=1$, and $x$ is periodic.
Therefore, $G'$ is torsion. Since $D$ is locally finite, $G'$ can be viewed as locally linear group, so, by Schur's Theorem \cite[p.154]{lam}, $G'$ is locally finite. The Amitsur's Theorem \cite[2.1.11, p.51]{shir-weh} shows that the only unsolvable finite subgroup of a division ring is $\SL(2,5)$. Hence, if $G'$ is not locally solvable, then $G'\cong \SL(2,5)$. Now, assume that $G'$ is locally solvable. Then, by Zaleskii's result \cite{zales}, $G'$ is solvable, so is $G$.
\end{proof}
\begin{lemma}\label{lem:2} Let $D$ be a division ring with center $F$, $G$ a subnormal subgroup of $D^*$ and assume that $M$ is a maximal subgroup of $G$. If $[M:Z(M)]<\infty$, then $M$ is abelian. 
\end{lemma} 
\begin{proof} Suppose that $M$ is non-abelian and $\{x_1, \ldots, x_m\}$ is a tranversal of $Z(M)$ in $M$. Denote by $E$ the subfield of $D$ generated by $Z(M)$ over the prime subfield $P$ of $F$, that is $E=P(Z(M))$, and put
	$$K=\Big\{\sum_{i=1}^ma_ix_i\vert a_i\in E\Big\}.$$
	Clearly, $K$ is a subring of $D$. Moreover, $K$ is a division subring of $D$ because it is a finite dimensional vector space over its center $E\subseteq Z(K)$. Also, we have $C_K(M)=Z(K)$. Since $M$ is a maximal subgroup of $G$, either $M=K^*\cap G$ or $G\subseteq K^*$. 
	
	\bigskip
	\textit{Case 1:} $M=K^*\cap G.$
	
	In this case, $M$ is subnormal in $K^*$. For $H=\ts{x_1, \ldots, x_m}$, we have $M=HZ(M)$, hence $M'\leq H$, so $H$ is normal in $M$. Therefore, $H$ is a finitely generated subnormal subgroup of $K^*$, and by \cite[Theorem 1]{fg}, $H\subseteq Z(K)$. In particular, $H$ is abelian, and this implies that $M$ is abelian too, but this is a contradiction because by the assumption, $M$ is non-abelian. 
	
	\bigskip
	\textit{Case 1:} $G\subseteq K^*.$
	
	By Stuth's theorem, either $K\subseteq F$ or $K=D$. If $K\subseteq F$, then $M\subseteq F$ that is a contradiction. Hence, $K=D$, so $D$ is centrally finite and $C_D(M)=Z(D)=F$. Take an arbitrary element $x\in G\setminus M$ and set $H=\ts{x, x_1, \ldots, x_m}$. By maximality of $M$ in $G$, we have $G=HZ(M)$. Hence, $G'\subseteq H$, so $H$ is normal in $G$, and consequently, $H$ is a finitely generated subnormal subgroup of $D^*$. By \cite[Theorem 1]{fg}, $H\subseteq F$. It follows $M\subseteq F$ that is impossible.
\end{proof} 

\begin{theorem}{\rm (\cite[Theorem 3.4]{hai-tu})} Let $D$ be a locally finite division ring with center $F$ and $G$ a subnormal subgroup of $D^*$. Assume that $M$ is  a non-abelian maximal subgroup of $G$. If $M$ contains no non-cyclic free subgroups, then $[D:F]<\infty$, $F(M)=D$, and there exists a maximal subfield $K$ of $D$ such that $K/F$ is a Galois extension, $\mathrm{Gal}(K/F)\cong M/K^*\cap G$ is a finite simple group, and $K^*\cap G$ is the $FC$-center. Also, $K^*\cap G$ is  the Fitting subgroup of $M$.
\end{theorem}
\begin{proof} Assume that $M$ does not contain non-cyclic free subgroups. We claim that $F(M)=D$. Indeed, if $F(M)\neq D$, then, by Stuth's theorem, $F(M)$ is not normalized by $G$. So, by maximality of $M$ in $G$, we have $F(M)\cap G=M$. Thus, $M$ is a non-abelian subnormal subgroup of $F(M)^*$ containing no non-cyclic free subgroup that contradicts to \cite[Theorem 11]{hai-ngoc}. Therefore, $F(M)=D$ as we have claimed. From the last equality, it follows $Z(M)=M\cap F$.
Since $D$ is locally finite, $M$ is locally linear group. So, by Tit's Alternative \cite{tits}, $M$ is locally solvable-by-finite. By \cite[3.3.9, p.103]{shir-weh}, $M$ is (locally solvable)-by-(locally finite). Note that the only unipotent element of a division ring is $1$, so, by \cite[3.3.2, p.96]{shir-weh}, $M$ is solvable-by-(locally finite). Therefore, there exists a solvable normal subgroup $H$ of $M$ such that $M/H$ is locally finite. By \cite[Theorem 5.5.1, Part a), p.199]{shir-weh}, there exists an abelian normal subgroup $A$ of $M$ such that $H/A$ is locally finite. Hence, by \cite[14.3.1, p.429]{rob}, $M/A$ is locally finite. Thus, $M$ is abelian-by-(locally finite). 

We claim that there exists an abelian normal subgroup of $M$ which is not contained in $F$. Indeed, if $A$ is not contained in $F$, then $A$ is such a subgroup. Otherwise,  $A\subseteq F$ and it follows that  $M$ is radical over $F$ because $M/A$ is locally finite. By Lemma \ref{lem:1}, $M'$ is locally finite. 

Now, let us consider three possible cases.

\bigskip
\textit{Case 1:} $\Char(F)>0.$

\bigskip

Since $M'$ is locally finite, $M'$ is abelian. Denotes by $B$ the maximal abelian normal subgroup of $M$ containing $M'$. The proof of \cite[Theorem 3.3 (i)]{hai-tu} shows in particular that $B$ is a normal subgroup of $M$ which is not contained in $F$.

\bigskip
\textit{Case 2:} $\Char(F)=0$ and $[D:F]<\infty$.

\bigskip	

By Tits' Alternative \cite{tits}, there exists a normal solvable subgroup $S$ of $M$ such that $M/S$ is finite. By \cite[Lemma 3]{free}, $S$ contains a normal abelian subgroup $H$ such that $S/H$ is finite. Then, $B=Core_M(H)$ is a normal abelian subgroup of $M$ of finite index. If $B\subseteq F$, then $B\subseteq Z(M)$, so it follows that $[M:Z(M)]<\infty$. Hence, in view of Lemma \ref{lem:2}, $M$ is abelian that is a contradiction. 

\bigskip
\textit{Case 3:} $\Char(F)=0$ and $[D:F]=\infty.$

\bigskip

Assume that $M'$ is not solvable. By Lemma \ref{lem:1}, $M'\cong \SL(2,5)$. If $F(M')\cap G\subseteq M$, then $M'\triangleleft F(M')\cap G$ and this implies that $M'$ is subnormal subgroup of $F(M')$. Since $M'$ contains no non-cyclic free subgroup, in view of \cite[Theorem~11]{hai-ngoc}, $M'$ is contained in the center of $F(M')$. In particular, $M'$ is abelian that is impossible because $M'\cong \SL(2,5)$. Thus, $F(M')\cap G\not\subseteq M$. Since   $M$ is maximal in $G$, $(F(M')\cap G)M=G$, so $G$ normalizes $F(M')^*$. By Stuth's theorem, we have $F(M')=D$ that entails $[D:F]<\infty$, a contradiction.
Therefore, $M'$ is  solvable. Let $r$ be a natural number such that $M^{(r)}\not\subseteq F$ and $M^{(r+1)}\subseteq F$. Then, $M^{(r)}$ is nilpotent, so in view of \cite[2.5.2, p. 73]{shir-weh}, there is an abelian normal subgroup $B$ of $M^{(r)}$ of finite index. Now, consider the division subring $D_1=F(M^{(r)})$ of $D$. Clearly, $M^{(r)}$ normalizes $D_1$. Suppose that $G$ also normalizes $D_1$. Since $M^{(r)}\not\subseteq F$, by Stuth'theorem, we have $F(M^{(r)})=D_1=D$. If $\{x_1, \ldots, x_m\}$ is a tranversal of $B$ in $M$, then $F(M^{(r)})=\sum_{i=1}^mx_iF(B)$. This implies the right dimension $[D:F(B)]_r$ is finite. By Double Centralizer Theorem, $[D:F]<\infty$, a contradiction. Hence, $D_1$ is not normalized by $G$. Then, by maximality of $M$ in $G$, we have $D_1\cap G\subseteq M$. This implies that $M^{(r)}$ is subnormal in $D_1$, so $M^{(r)}$ is a nilpotent subnormal subgroup of $D_1$. By Lemma \cite[Lemma 3.1]{hai-tu}, $M^{(r)}\subseteq Z(D_1)$, so $M^{(r)}$ is an abelian normal subgroup of $M$ which is not contained in $F$. Thus, we have proved that in all three cases, there exists an  abelian normal subgroup of $M$ which is not contained in $F$, as claimed. 
The claim is proved.
	
Now, assume that $C$ is an abelian normal subgroup of $M$ which is not contained in $F$. Then, $C\not\subseteq Z(M)$ because as we have noted above, $Z(M)=M\cap F$. Take an element $\alpha\in C\setminus Z(M)$. Since $Z(M)=M\cap F$ and $M/Z(M)$ is locally finite, $\alpha$ is algebraic over $F$. Set $K=F(\alpha^M)$ and $H=C_D(K)$. The condition $C\triangleleft M$ implies that the elements of $\alpha^M$ in the field $F(\alpha^M)$ have the same minimal polynomial over $F$, so $|\alpha^M|<\infty$, hence $\alpha$ is an $FC$-element. Therefore, we can apply \cite[Theorem~3.2]{hai-tu} for $\alpha, K$ and $H$. In particular, $H^*\cap G$ is a subgroup of $M$. Clearly, $H^*\cap G$ is subnormal in $H^*$. So, by \cite[Theorem 11]{hai-ngoc}, we have
$$H^*\cap G\subseteq Z(H)=C_H(H)\subseteq C_D(H)=C_D(C_D(K))=K.$$
Therefore, all the requirements of  \cite[Theorem 3.2]{hai-tu} are satisfied. The proof of the theorem is now complete.
\end{proof}

\end{document}